\documentclass[12pt,a4paper]{article}
\usepackage{amsmath,amssymb,amsthm}
\usepackage{amssymb,latexsym, bbm,comment}
\usepackage{graphicx}
\usepackage{makeidx}

\newtheorem{theorem}{Theorem}[section]

\newcommand{\R}{\mathbb{R}}
\newcommand{\Rd}{\mathbb{R}^d}
\newcommand{\nN}{n \in \mathbb{N}}
\newcommand{\N}{\mathbb{N}}
\newcommand{\C}{\mathbb{C}}
\newcommand{\E}{\mathbb{E}}

\newcommand{\cadlag}{c\`adl\`ag}

\newcommand{\tr}{\mbox{tr}}

\newcommand{\st}{\stackrel{d}{=}}
\newcommand{\bean}{\begin{eqnarray*}}
\newcommand{\eean}{\end{eqnarray*}}

\newcommand{\la}{\langle}
\newcommand{\ra}{\rangle}

\begin{document}

\date{}

\title{Infinite Dimensional Ornstein-Uhlenbeck Processes Driven by L\'{e}vy Processes}

\author{David Applebaum, \\ School of Mathematics and Statistics,\\ University of
Sheffield,\\ Hicks Building, Hounsfield Road,\\ Sheffield,
England, S3 7RH\\ ~~~~~~~\\e-mail: D.Applebaum@sheffield.ac.uk}

\maketitle

\begin{abstract}

We review the probabilistic properties of Ornstein-Uhlenbeck processes in Hilbert spaces driven by L\'{e}vy processes. The emphasis is on the different contexts in which these processes arise, such as stochastic partial differential equations, continuous-state branching processes, generalised Mehler semigroups and operator self-decomposable distributions. We also examine generalisations to the case where the driving noise is cylindrical.

\end{abstract}

\section{Introduction}

The physical phenomenon of Brownian motion is well described by the stochastic process that bears its name, which originated in the work of Einstein, Bachelier, Smoluchowski and Wiener, provided the particle moves ``freely'' in its surrounding medium, subject only to the random force of molecular bombardment. If the viscous properties of the medium are also taken into account, Ornstein and Uhlenbeck \cite{OrUh} proposed that the velocity $v(t)$ at time $t$ of a diffusing particle of mass $m$ should satisfy the Langevin equation:
\begin{equation} \label{langpre}
m\frac{dv}{dt} + \beta v = F(t),
\end{equation}
where $\beta$ is the viscosity coefficient, and $F$ is a random force acting on the particle. Doob \cite{Doob} rewrote (\ref{langpre}) as an SDE in which the formal differential $F(t)dt$ was replaced by the stochastic differential $dB(t)$ of Brownian motion. Its solution (with $m= 1$, for convenience) is the prototypical {\it Ornstein-Uhlenbeck process},
$$ v(t) = e^{-\beta t}v(0) + \int_{0}^{t}e^{-\beta(t-s)}dB(s),$$
wherein the right-hand-side is the standard Wiener stochastic integral, soon to be extensively generalised by It\^{o}. The process $(v(t), t \geq 0)$ is both Gaussian and Markov, and if we choose $v(0)$ to have the limiting distribution $N(0, 1/2\beta)$, then it is also stationary.

From a probabilistic point of view, it is natural to consider different driving noises, and also to work in a multivariate framework. A number of authors have investigated the case where $B$ is generalised to be an arbitrary L\'{e}vy process, and the corresponding generalised Ornstein-Uhlenbeck processes are interesting from both a theoretical and applied viewpoint (see e.g. \cite{Appbk} pp.241-3 and references therein). Most of the theoretically interesting properties in the finite-dimensional case, such as the structural relationship with the class of self-decomposable distributions, will re-emerge in the the infinite-dimensional framework that we consider below; among the most important areas of application we mention volatility modelling in mathematical finance \cite{BNS, KLM}, and physical models of anomalous diffusion \cite{JMF}.

In this review, we will be concerned solely with infinite-dimensional generalisations where the noise is a general L\'{e}vy process taking values in a real Hilbert space. The rationale for this choice, is that this class of processes is very rich, and that they lie at the intersection of a number of interesting areas in probability theory and stochastic analysis. In particular this article will consider their role in the study of stochastic partial differential equations (SPDEs), generalised Mehler semigroups, operator self-decomposable distributions, continuous-state branching processes, and cylindrical measures and processes. Some of these connections have been investigated by other authors, including the current one, see e.g. \cite{App1, BRS, Li, Jur1}, but it seems to be a useful endeavour to discuss them all within a single article, and to give readers some pointers to the large and growing literature in these active areas of research.

We will not consider the case herein where the noise is purely Gaussian. This is because that theory has been very well developed, and the interested reader can find full accounts in the monographs \cite{dPZ, dPZ2}, and references therein. We also resisted the temptation to generalise further to Banach space-valued processes, as that theory has not yet reached such a mature stage as the Hilbert space one.

The plan of the paper is as follows. After some preliminaries on L\'{e}vy processes and stochastic integrals in section 2, we define infinite-dimensional Orstein-Uhlenbeck processes in section 3, and see how they arise naturally in the study of SPDEs with additive L\'{e}vy noise. In section 4, we collect together some of the basic probabilistic properties of these processes. Mehler semigroups of operators, skew-convolution semigroups of measures and their generalisations are the topic of section 5. In section 6, we turn to invariant measures and operator self-decomposability. Finally in section 7, we look at various approaches to defining Ornstein-Uhlenbeck processes driven by cylindrical L\'{e}vy noise.

\vspace{5pt}

{\it Preliminaries and Notation.} Throughout this work $H$ is a real separable Hilbert space, with inner product $\la \cdot, \cdot \ra$ and norm $||\cdot||$, ${\cal B}(H)$ is its Borel $\sigma$-algebra, and  ${\cal M}_{1}(H)$ is the space of all Borel probability measures defined on $H$. We denote by $B_{b}(H)$ the Banach space (with respect to the supremum norm) of all real-valued, bounded Borel measurable functions defined on $H$, and by $C_{b}(H)$, its closed subspace of bounded continuous functions.  The open ball of radius $1$ in $H$, centered at the origin, will always be denoted by $B_{1}$. A positive self-adjoint linear operator $T$ in $H$ is {\it trace-class} if its trace $\tr(T) = \sum_{n=1}^{\infty} \la Te_{n}, e_{n} \ra < \infty$, for some (and hence all) complete orthonormal basis $(e_{n}, \nN)$ in $H$. The algebra of all bounded linear operators on $H$ will be denoted by $L(H)$.

The {\it Fourier transform} (or characteristic function) of $\mu \in {\cal M}_{1}(H)$ is the bounded continuous mapping $\widehat{\mu}: H \rightarrow \C$ defined for all $u \in H$ by
$$ \widehat{\mu}(u) = \int_{H}e^{i \la u, x \ra}\mu(dx).$$
The {\it convolution} $\mu_{1} * \mu_{2}$ of $\mu_{1}, \mu_{2} \in {\cal M}_{1}(H)$, is the unique Borel probability measure for which
$$ \int_{H}f(x)(\mu_{1} * \mu_{2})(dx) = \int_{H}\int_{H}f(x+y)\mu_{1}(dx)\mu_{2}(dy),$$
for all $f \in B_{b}(H)$. Note that $\mu_{1} * \mu_{2} = \mu_{2} * \mu_{1},$ and for all $u \in H,$ $$\widehat{\mu_{1} * \mu_{2}}(u) = \widehat{\mu_{1}}(u)\widehat{\mu_{2}}(u).$$
Suppose that $Q$ is a probability kernel, i.e. a mapping $H \times {\cal B}(H) \rightarrow [0,1]$, so that for fixed $A \in {\cal B}(H)$, the mapping $x \rightarrow Q(x, A)$ is measurable, and for fixed $x \in H, Q(x, \cdot) \in {\cal M}_{1}(H)$. If $\rho \in {\cal M}_{1}(H)$, then $Q\rho$ is defined to be the unique measure in ${\cal M}_{1}(H)$ for which
$$ \int_{H}f(x)(Q\rho)(dx) = \int_{H}\int_{H}f(y)Q(x, dy)\rho(dx),$$
for all $f \in B_{b}(H)$. If $T:H \rightarrow H$ is a bounded linear operator, and $Q(x, \cdot) = \delta_{Tx}$ (where $\delta_{y}$ is the usual Dirac mass at $y \in H$), then $Q\rho = \rho \circ T^{-1}$, which we also write succinctly as $T\rho$.

\section{L\'{e}vy Processes and Stochastic Integration in Hilbert Space}

Let $(\Omega, {\cal F}, P)$ be a probability space that is equipped with a filtration $({\cal F}_{t}, t \geq 0)$ , which satisfies the ``usual hypotheses'' of right-continuity and completeness. Let $H$ be a real separable Hilbert space. A {\it L\'{e}vy process} is an adapted $H$-valued process $L = (L(t), t \geq 0)$ for which
\begin{itemize}
\item $L(0) = 0$ (a.s.),
\item $L$ is stochastically continuous,
\item $L$ has stationary and independent increments, where the latter is in the strong sense that $L(t) - L(s) $ is independent of ${\cal F}_{s}$, for all $0 \leq s < t < \infty$,
\item $L$ has \cadlag~paths.
\end{itemize}
Then  $L$ has the following L\'{e}vy-It\^{o} decomposition (see \cite{AR,RvG}), for all $t \geq 0$:

\begin{equation} \label{LIdec}
L(t) = bt + B_{Q}(t) + \int_{||x|| \leq 1}x\tilde{N}(t, dx) + \int_{||x|| > 1}xN(t, dx),
\end{equation}
wherein
\begin{itemize}
\item The vector $b \in H$,
\item The process $B_{Q} = (B_{Q}(t), t \geq 0)$ is a Brownian motion taking values in $H$ with covariance operator $Q$, so that for all $s,t \geq 0, \psi, \phi \in H$,
$$ \E(\la \psi, B_{Q}(s) \ra \la \phi, B_{Q}(t) \ra ) =  \la Q\psi,  \phi \ra(s \wedge t),$$ where $Q$ is a positive, symmetric, trace-class operator,
\item $N$ is a Poisson random measure on $[0, \infty) \times H$, which is independent of $B_{Q}$, and has compensator:
$$ \tilde{N}(dt, dx) = N(dt, dx) - dt \nu(dx),$$
where $\nu$ is a L\'{e}vy measure on $H$, i.e.
$$ \nu(\{0\}) = 0~\mbox{and}~\int_{H}(||x||^{2} \wedge 1)\nu(dx) < \infty.$$
\end{itemize}

From this we can deduce the L\'{e}vy-Khintchine formula, for all $ t \geq 0, u \in H$,

$$ \E(e^{i \la u, L(t) \ra}) = e^{t \eta(u)},$$
where the {\it characteristic exponent} $\eta: H \rightarrow \C$ is given by
\begin{equation} \label{LK}
\eta(u) = i \la u, b \ra - \frac{1}{2}\la Qu, u \ra \\
 +  \int_{H}(e^{i \la u, y \ra} - 1 - i \la u, y \ra {\bf 1}_{B_{1}}(y)) \nu(dy).
\end{equation}

The law of $L(t)$ is uniquely determined by its {\it characteristics} $(b, Q, \nu)$.

The process $L$ is a martingale, and will be called a {\it L\'{e}vy martingale}, if and only if $\int_{|x| > 1}|x| \nu(dx) < \infty$ and $b = - \int_{|x| > 1}x \nu(dx)$ (see e.g. \cite{Appbk} p.133) in which case we may write (\ref{LIdec}) as

$$ L(t) = B_{Q}(t) + \int_{H}x\tilde{N}(t, dx).$$

If $L(t) = L_{M}(t) + bt$, where $L_{M}$ is a L\'{e}vy martingale, we call $L$ a {\it L\'{e}vy martingale with drift}.

If $F = (F(t), t \geq 0)$ is a suitably regular adapted, bounded operator valued process in $H$, we can makes sense of stochastic integrals $\int_{0}^{t}F(t)dL(t)$, either by using the general approach of \cite{MePe} (see in particular, sections 1.2 and 6.14), or by using the L\'{e}vy-It\^{o} decomposition (\ref{LIdec}), as in \cite{App1} and Chapter 8 of \cite{PZ}, and defining:

\bean \int_{0}^{t}F(s)dL(s) & = & \int_{0}^{t}F(s)b~ds + \int_{0}^{t}F(s)dB_{Q}(s)\\
 & + & \int_{0}^{t}\int_{||x|| \leq 1}F(s)x\tilde{N}(t, dx) + \int_{0}^{t}\int_{||x|| > 1}F(s)xN(t, dx).
\eean

Here the first integral is a Bochner integral, and the last is a random sum. The middle two may be defined  by using stochastic integration against martingale-valued measures. In this paper we will only be concerned with the case where $F$ is deterministic. In that case, we may also define $\int_{0}^{t}F(s)dL(s)$ using integration by parts, as in \cite{JV}.

\vspace{5pt}

{\bf Note.} A Borel probability measure on $H$ is infinitely divisible if and only if its characteristic exponent is given by (\ref{LK}); see Theorem 4.10, pp.181-2 in \cite{Par1}.

\section{Ornstein-Uhlenbeck Processes and Stochastic Partial Differential Equations}

Let $(S(t), t \geq 0)$ be a $C_{0}$-semigroup acting in $H$ (i.e. a strongly continuous one-parameter semigroup of bounded, linear operators on $H$). Let $Y_{0}$ be a fixed ${\cal F}_{0}$ measurable random variable, and $L$ be an $H$-valued L\'{e}vy process. The {\it Ornstein-Uhlenbeck process} (or OU process, for short) associated with these data is the adapted process $Y = (Y(t), t \geq 0)$ defined for $t \geq 0$ by
\begin{equation} \label{OU}
Y(t) = S(t)Y_{0} + \int_{0}^{t}S(t-r)dL(r).
\end{equation}

Observe that if $H = \R$, then we must have $S(t) = e^{\lambda t}$ for some $\lambda \in \R$, and we recapture the usual L\'{e}vy-driven one-dimensional Ornstein-Uhlenbeck process (see e.g. \cite{Appbk} pp.237-43, and references therein). When $H$ is finite dimensional we may consider matrix semigroups of the form $S(t) = e^{-tA} = \sum_{n=0}^{\infty}(-1)^{n}\frac{A^{n}}{n!}$. Such objects are discussed from a distributional perspective in \cite{Sa} pp.106-14. From now on we will always assume that $H$ is infinite-dimensional.

The stochastic integral in (\ref{OU}) always make sense using the procedures that we have sketched (see also \cite{CM}). The process $t \rightarrow \int_{0}^{t}S(t-r)dL(r)$ is called a {\it stochastic convolution}. It is a semimartingale, but in contrast to the finite-dimensional case, when $L$ is a L\'{e}vy martingale, the stochastic convolution is not, in general, the product of a martingale with a deterministic process. This is because we cannot decompose the semigroup as ``$S(t-r) = S(t)S(-r)$''.  This causes problems in studying the time-regularity of the OU process, which will be discussed in section 4.

Motivation for studying the OU process comes from the study of {\it stochastic partial differential equations}, or SPDEs for short, that are driven by L\'{e}vy noise. Consider, for example the heat equation in $\Rd$,
\begin{equation} \label{he}
\frac{\partial u}{\partial t} = \sum_{j=1}^{d}\frac{\partial^{2} u}{\partial x_{j}^{2}},
\end{equation}
with initial condition $u(0) = u_{0} \in L^{2}(\Rd)$. One approach to randomizing this equation would be to introduce a suitable two-parameter L\'{e}vy sheet $(L(x,t), x \in \R^{d}, t \geq 0)$ and seek to give meaning to
the equation
$$ \frac{\partial Y(t,x)}{\partial t} = \sum_{j=1}^{n}\frac{\partial^{2} Y(t,x)}{\partial x_{j}^{2}} + \dot{L}(x,t),$$
whose solution would be a random field on $[0, \infty) \times \Rd$.

This procedure is implemented in \cite{ApWu}, by extending the methodology developed by Walsh \cite{Wa} for the case of Gaussian noise. An alternative approach is to recognise that the Laplacian $\Delta = \sum_{j=1}^{d}\frac{\partial^{2} u}{\partial x_{j}^{2}}$ is the infinitesimal generator of the heat semigroup $(S_{h}(t), t \geq 0)$ in $L^{2}(\Rd)$. Then the solution of the heat equation (\ref{he}) is given by $u(t) = S_{h}(t)u_{0}$ in $L^{2}(\Rd)$. From this point of view, we may introduce L\'{e}vy noise by using the Hilbert-space valued L\'{e}vy process $L$ and write the SPDE as an infinite-dimensional,  It\^{o}-sense stochastic differential equation

\begin{equation} \label{she}
 dY(t) = \Delta Y(t)dt + dL(t).
\end{equation}

Next we generalise, and appreciate that there is nothing special from the point of view of It\^{o} calculus, about the semigroup $(S_{h}(t), t \geq 0)$ acting in $L^{2}(\Rd)$. In fact we may consider an arbitrary $C_{0}$-semigroup $(S(t), t \geq 0)$ having infinitesimal generator $A$, and acting in a Hilbert space $H$. Then we seek to make sense of solutions to:

\begin{equation} \label{Lang}
 dY(t) = A Y(t)dt + dL(t).
\end{equation}

We regard equation (\ref{Lang}) as an infinite-dimensional {\it Langevin equation}. It has a unique mild solution which is given by the variation of constants formula, and this is precisely the OU process (\ref{OU}). As is shown in Theorems 9.15 and 9.29 of \cite{PZ}, or in \cite{App1}, this is also the unique weak solution in that for all $u \in \mbox{Dom}(A^{*}), t \geq 0$, with probability $1$,
\begin{equation} \label{weak}
\la u, Y(t) - Y_{0} \ra = \la u, L(t) \ra + \int_{0}^{t} \la A^{*}u, Y(s) \ra ds.
\end{equation}

Equation (\ref{Lang}) is the simplest SPDE with additive L\'{e}vy noise. More generally, one may consider more complicated equations with multiplicative noise of the form:
\begin{equation} \label{SDEm}
dY(t) = (AY(t) + F(t,Y(t-))) dt + G(t,Y(t-))dL(t).
\end{equation}

For a comprehensive introduction to the study of these equations, see \cite{PZ}. The equation (\ref{SDEm}) is called {\it semilinear} if $G(t,Y(t-)) = G$ is a bounded linear operator, and in this case in particular, knowledge of the case $F \equiv 0$ (i.e. the Ornstein-Uhlenbeck process), may be an important first step in obtaining information about the solution to the more general equation, see e.g. \cite{PZ1,APMS}. The two different approaches to solving SPDEs that we have presented above, using on the one hand, two-parameter ``space-time white noise'', and on the other hand, infinite-dimensional processes, have each generated a considerable literature; nonetheless there are important cases where they both give rise to the same solution, see e.g. \cite{DQ}.

\vspace{5pt}

{\bf Note.} A rather straightforward generalisation of (\ref{Lang}) considers equations of the form
$$ dY(t) = AY(t) dt + BdL(t),$$
where $L$ takes values in a different Hilbert space $H_{1}$, and $B$ is a bounded linear operator from $H_{1}$ to $H$. For simplicity, we will always work with (\ref{Lang}) in this article; readers who wish to consider the extended case can easily make the required minor adjustments.

\section{Basic Properties of Ornstein-Uhlenbeck Processes}

It follows from elementary properties of stochastic integrals with respect to L\'{e}vy processes that $(Y(t) - Y_{0}, t \geq 0)$ is an additive process, i.e. a stochastically continuous process, which vanishes at zero (a.s.), and has independent increments. Hence, for $t > 0$, the random variable $Y(t) - Y_{0}$ is infinitely divisible. From our assumptions, $Y(t) - Y_{0}$ is independent of $Y_{0}$, and so, if $Y_{0}$ is infinitely divisible, then so is $Y(t)$. If $Y_{0}$ is infinitely divisible with characteristic exponent $\eta_{0}$, we have (see \cite{CM, App2}), for all $u \in H$,
$$ \E(e^{i \la u , Y(t) \ra}) = \exp{ \left\{\eta_{0}(S(t)^{*}(u)) + \int_{0}^{t}\eta(S(r)^{*}(u))dr \right\}}
.$$

Furthermore, if $Y_{0}$ has characteristics $(b_{0}, Q_{0}, \nu_{0})$, then $Y(t)$ has characteristics $(b_{t}, Q_{t}, \nu_{t})$ where:

\bean b_{t} & = & S(t)b_{0} + \int_{H}S(t)x[{\bf 1}_{B_{1}}(S(t)x) - {\bf 1}_{B_{1}}(x)]\nu_{0}(dx)\\
& + & \int_{0}^{t}S(r)bdr + \int_{0}^{t}\int_{H}S(r)x[{\bf 1}_{B_{1}}(S(r)x) - {\bf 1}_{B_{1}}(x)]\nu(dx)dr, \eean

$$ Q_{t} = S(t)Q_{0}S(t)^{*} + \int_{0}^{t}S(r)QS(r)^{*}dr,$$

$$ \nu_{t}(B) = \nu_{0}(S(t)^{-1}(B)) + \int_{0}^{t}\nu_{0}(S(r)^{-1}(B))dr, $$

for each $B \in {\cal B}(H)$.

In the frequently encountered case where $Y_{0} = y_{0}$ (a.s.), with fixed $y_{0} \in H$, then $\eta_{0}(S(t)^{*}u) = \la u, S(t)y_{0} \ra$.

Necessary and sufficient conditions for the OU process $Y$ to have \cadlag~paths have not yet been established, and this seems to be a difficult problem, in general. Two sufficient conditions are described in \cite{PZ}, pp.156-61. In both cases we require the process $L$ to be a square-integrable L\'{e}vy martingale with drift.

\begin{enumerate}

\item {\it The Kotelenez approach.} This utilises a generalisation of Doob's submartingale inequality to stochastic convolution with respect to martingales. In our case, the conclusion is that the stochastic convolution (and hence the OU process) has a \cadlag~version if $(S(t), t \geq 0)$ is a generalised contraction semigroup, i.e. there exists $\beta > 0$ so that $||S(t)|| \leq e^{\beta t}$ for all $t \geq 0$.

\item {\it The Hausenblas-Seidler approach.} This requires $(S(t), t \geq 0)$ to be a contraction semigroup. It utilises the Nagy-Foias theory of dilations of such semigroups to embed $H$ as a closed subspace of a larger Hilbert space $H_{1}$, and assert that there is a strongly continuous unitary group $(U(t), t \in \R)$ in that larger space so that $S(t) = PU(t)$ for all $t \geq 0$, where $P$ is the orthogonal projection of $H_{1}$ onto $H$. Then the stochastic convolution has a \cadlag~version. A key feature of the proof is to work in $H_{1}$, and use the group property to write
    $$ \int_{0}^{t}U(t-s)dL(s) = U(t)\int_{0}^{t}U(-s)dL(s).$$

\end{enumerate}

\noindent Since it is the solution to an SDE, $(Y(t), t \geq 0)$ is a Markov process, and we may compute the action of the corresponding transition semigroup $(P_{t}, t \geq 0)$ on the space $B_{b}(H)$ from (\ref{OU}) to obtain
for each $t \geq 0, f \in B_{b}(H), x \in H$,

\begin{eqnarray} \label{Me1}
P_{t}f(x) & = & \E(f(Y(t))| Y_{0} = x) \nonumber \\
& = & \int_{0}^{t}f(S(t)x + y)\mu_{t}(dy),
\end{eqnarray}
where $\mu_{t}$ is the law of $\int_{0}^{t}S(t-r)dL(r)$, or equivalently, by stationary increments of $L$, the law of $\int_{0}^{t}S(r)dL(r)$. Note that the semigroup $(P_{t}, t \geq 0)$ is not, in general, strongly continuous on $C_{b}(H)$ or its subspace of bounded, uniformly continuous real-valued functions on $H$; this is discussed for the Gaussian case in \cite{dPZ2}, p.111, and the same argument works in the L\'{e}vy case. It is however always quasi-equicontinuous with respect to a weaker topology, called the mixed topology, in $C_{b}(H)$. The semigroup $(P_{t}, t \geq 0)$ then has an infinitesimal generator ${\cal L}$ (in this generalised sense), which has the following representation for functions $f$ in a certain dense domain (see \cite{App2} for details):
\bean
{\cal L }f(x) & = & \langle A^{*}(Df)(x), x \rangle + \langle (Df)(x), b \rangle +
\frac{1}{2}\tr((D^{2}f)(x)Q) \nonumber \\
 & + & \int_{H-\{0\}}[ f(x+y) - f(x) - \langle (Df)(x), y \rangle
{\bf 1}_{B_{1}}(y)]\nu(dy),
\eean

\noindent where $x \in H$, and $D$ is the Fr\'{e}chet derivative. Pseudo-differential operator representations of ${\cal L}$ have been obtained in \cite{LR1}, and in Proposition 4.1 of \cite{App2}.

The measures $(\mu_{t}, t \geq 0)$ have an interesting property, which will play an important role in the sequel. Observe that for all $r, t \geq 0$:
$$ \int_{0}^{r+t}S(v)dL(v)  =  \int_{0}^{r}S(v)dL(v)  + \int_{r}^{r+t}S(v)dL(v),$$ and that the two stochastic integrals on the right hand side of this equation are independent.
Furthermore, by stationary increments of $L$, and  the semigroup property we have
$$ \int_{r}^{r+t}S(v)dL(v) \st  \int_{0}^{t}S(r+v)dL(v)= S(r)\int_{0}^{t}S(v)dL(v). $$

We conclude that

\begin{equation}  \label{scs1}
\mu_{r+t} = \mu_{r} * S(r)\mu_{t}.
\end{equation}

It is also not difficult to verify that for all $u \in H, t \geq 0$,

\begin{equation} \label{FTsc}
\widehat{\mu_{t}}(u) = \exp{\left\{\int_{0}^{t}\eta(S(r)^{*}u)dr \right\}}.
\end{equation}

\section{Mehler Semigroups and Skew-Convolution Semigroups}

In this section we look at a more abstract approach to some of the ideas we've encountered in the last section. Let us suppose that we are given, as above, a $C_{0}$-semigroup $(S(t), t \geq 0)$ on $H$, and also a family $(\rho_{t}, t \geq 0)$ of Borel probability measures on $H$, with $\rho_{0} = \delta_{0}$. Define the linear operators $(T_{t}, t \geq 0)$ on $B_{b}(H)$ by

\begin{equation} \label{Me2}
T_{t}f(x) = \int_{0}^{t}f(S(t)x + y)\rho_{t}(dy),
\end{equation}

\noindent for each $t \geq 0, f \in B_{b}(H), x \in H$.

A natural question to ask is, when does $(T_{t}, t \geq 0)$ satisfy the semigroup property, i.e. $T_{s+t} = T_{s}T_{t}$, for all $s,t \geq 0$? As is shown in Proposition 2.2 of \cite{BRS}, this will hold if and only if
$(\rho_{t}, t \geq 0)$ is a {\it skew-convolution semigroup}, i.e. for all $r,t \geq 0$,

\begin{equation}  \label{scs2}
\rho_{r+t} = \rho_{r} * S(r)\rho_{t},
\end{equation}

\noindent or equivalently, for each $u \in H$,

$$ \widehat{\rho_{r+t}}(u) = \widehat{\rho_{r}}(u)\widehat{\rho_{t}}(S(r)^{*}u).$$

When (\ref{scs2}) holds, we call $(T_{t}, t \geq 0)$ a (generalised) {\it Mehler semigroup} as (\ref{Me2}) generalises the classical Mehler formula for one-dimensional Gaussian Ornstein-Uhlenbeck processes (see e.g. \cite{Appbk}, p.405). Mehler semigroups have been extensively studied in recent years, as analytic objects in their own right, particularly by M.R\"{o}ckner and his collaborators, and we draw the reader's attention to \cite{BRS,FR,LR1,LR2,RW,ORW}.  Skew convolution semigroups have appeared in other contexts (see below) and are also called {\it measure-valued cocycles} -- see e.g. \cite{Jur1}.

We have seen in the last section, that Mehler semigroups appear as the transition semigroups of Hilbert space-valued OU processes. We might ask if the converse is valid, and if, starting from a Mehler semigroup, we can construct a \cadlag~OU process for which it is the transition semigroup. The answer is affirmative and we will sketch the idea of the proof, drawing on the accounts in \cite{BRS}, \cite{FR}, and section 11.5 of \cite{Li}.

Following \cite{FR}, we assume that for all $u \in H$, the mapping $t \rightarrow \widehat{\rho_{t}}(u)$ is absolutely continuous on $[0, \infty)$, differentiable at $t=0$, and defining $\eta(u) = \left.\frac{d}{dt}\widehat{\rho_{t}}(u)\right|_{t=0}$, that $t \rightarrow \eta(S(t)^{*}u)$ is locally integrable. It follows that (c.f. (\ref{FTsc}))
\begin{equation} \label{FTscm}
\widehat{\rho_{t}}(u) = \exp{\left\{\int_{0}^{t}\eta(S(r)^{*}u)dr\right\}}.
\end{equation}
Furthermore the mapping $\eta$ is negative definite, and we may define a weakly-continuous convolution semigroup of probability measures $(\xi_{t}, t \geq 0)$ on $H$ so that for all $t \geq 0$,
$$ \widehat{\xi_{t}}(u) = e^{t\eta(u)}.$$
Then we may obtain a \cadlag~L\'{e}vy process $L = (L(t), t \geq 0)$ defined on some probability space, and taking values in $H$, having characteristic exponent $\eta$ by using standard techniques (see e.g. Chapters 1 and 2 of \cite{Appbk}).

The next step is to construct a Hilbert space $\tilde{H}$, in which $H$ is continuously embedded, so that $(S(t), t \geq 0)$ extends to a $C_{0}$-semigroup $(\widetilde{S}(t), t \geq 0)$ in $\tilde{H}$ whose generator $\tilde{A}$ is an extension of $A$, with $H \subset \mbox{Dom}(\tilde{A})$. To see how to do this, we first note that there exists $M \geq 1$ and $c \geq 0$ so that for all $t \geq 0, ||S(t)|| \leq Me^{tc}$. Choose $\lambda > c$ and let $R_{\lambda} = (\lambda I - A)^{-1}$ be the corresponding resolvent of $A$. We define a new inner product $\la \cdot, \cdot \ra_{1}$ on $H$ by the prescription
$$ \la \phi, \psi \ra_{1} = \la R_{\lambda} \phi, R_{\lambda} \psi \ra,$$
and define $\tilde{H}$ to be the completion of $H$ in the corresponding norm. Now for all $\tilde{u} \in \tilde{H}$, we define

\begin{equation} \label{OUcons}
Y(t) = \widetilde{S}(t)\tilde{u} + L(t) + \int_{0}^{t}\widetilde{S}(t-s)\tilde{A}L(s)ds.
\end{equation}

By formal differentiation of (\ref{OUcons}), we see that (\ref{OU}) is satisfied by this process (with respect to the extended Hilbert space, and semigroup), with initial condition $Y(0) = \tilde{u}$. Observe that the construction of $\tilde{H}$ ensures that the integral in (\ref{OUcons}) makes sense as $L(s) \subseteq \mbox{Dom}(\tilde{A})$ for all $s \geq 0$. As pointed out by Li in \cite{Li}, we may take $\tilde{u} \in H$ to get an $H$-valued version of the process, but this may not have \cadlag~paths.

In the case where the skew-convolution semigroup is square-integrable, so that $\int_{H}||x||^{2}\rho_{t}(dx) < \infty$ for all $t \geq 0$, it is shown in \cite{Li}, section 11.4. that the regularity conditions assumed above to obtain (\ref{FTscm}) may be dropped; but the price we pay for this is that the measures $(\xi_{t}, t \geq 0)$ will live on a larger Hilbert space than $H$, consisting of the locally square-integrable entrance paths for the semigroup $(S(t), t \geq 0)$.

We now discuss an interesting extension of the notion of skew convolution semigroup, which is due to Dawson and Li \cite{DL0}. Let $(E, +)$ be an abelian Hausdorff semigroup, and for $t \geq 0$, let $Q_{t}: E \times {\cal B}(E) \rightarrow [0,1]$ be the transition kernels of a Markov process. We assume that the underlying process has a {\it branching property} so that for all $t \geq 0, x_{1}, x_{2} \in E$
\begin{equation} \label{branch}
Q_{t}(x_{1} + x_{2}, \cdot) = Q_{t}(x _{1}, \cdot) * Q_{t}(x_{2}, \cdot).
\end{equation}

We say that a family of probability measures $(\mu_{t}, t \geq 0)$ on $(E, {\cal B}(E))$  is a {\it generalised skew-convolution semigroup} if for all $s, t \geq o$,

\begin{equation} \label{genscs}
  \mu_{s+t} = Q_{t}\mu_{s} * \mu_{t}.
\end{equation}

Observe that (\ref{genscs}) coincides with (\ref{scs2}), when we take $E = H$ and $Q_{t}(x, \cdot) = \delta_{S(t)x}(\cdot).$  Now define for $t \geq 0, x \in E$, the kernels

\begin{equation} \label{newk}
Q_{t}^{\mu}(x, \cdot) = Q_{t}(x, \cdot) * \mu_{t}.
\end{equation}

The next result gives a generalisation of the Mehler semigroup/OU process circle of ideas to this broader context. We include the proof for the reader's convenience, as we haven't found one in the literature.

\begin{theorem} \label{trans}
$(Q_{t}^{\mu}, t \geq 0)$ is the transition kernel of a Markov process.
\end{theorem}

\begin{proof}
We must show that the Chapman-Kolomogorov equations are satisfied, i.e. that for all $f \in B_{b}(E), x \in E$,
$$ \int_{S}f(z)Q_{s+t}^{\mu}(x, dz) = \int_{S}\int_{S}f(z)Q_{t}^{\mu}(y, dz)Q_{s}^{\mu}(x, dy).$$
We begin with the right hand side of the last display and use (\ref{newk}) twice to obtain
\bean & & \int_{S}\int_{S}f(z)Q_{t}^{\mu}(y, dz)Q_{s}^{\mu}(x, dy)\\ & = & \int_{S}\int_{S}\int_{S}f(u+v)Q_{t}(y, du)\mu_{t}(dv)Q_{s}^{\mu}(x, dy)\\
& = & \int_{S}\int_{S}\int_{S}\int_{S} f(u+v)Q_{t}(w_{1}+w_{2}, du)\mu_{t}(dv)Q_{s}(x, dw_{1})\mu_{s}(dw_{2})\\
& = & \int_{S}\int_{S}\int_{S}\int_{S}\int_{S} f(u_{1} + u_{2} +v)\mu_{t}(dv) Q_{t}(w_{1}, du_{1}) Q_{t}(w_{2}, du_{2})Q_{s}(x, dw_{1})\mu_{s}(dw_{2})\\
& = & \int_{S}\int_{S}\int_{S} f(u+v)\mu_{s+t}(dv)Q_{t}(w_{1}, du)Q_{s}(x, dw_{1})\\
& = & \int_{S}\int_{S}f(u+v)\mu_{s+t}(dv)Q_{s+t}(x,du)\\
& = & \int_{S}f(z)Q_{s+t}^{\mu}(x, dz), \eean
where we have used (\ref{branch}), (\ref{genscs}) and Fubini's theorem.

\end{proof}

\vspace{5pt}

{\it Remark.} The following was pointed out to the author by Zenghu Li. Suppose that $K$ has a neutral element $e$, and that $Q_{t}(e, \cdot) = \delta_{e}(\cdot)$ for all $t \geq 0$. Then (\ref{genscs}) is also a necessary condition for $(Q_{t}^{\mu}, t \geq 0)$ to be a transition kernel. This follows from straightforward manipulations of the Chapman-Kolmogorov equation for $Q^{\mu}_{s+t}(e, \cdot)$.

\vspace{5pt}

There are two other important examples, apart from the Mehler semigroup case:
\begin{itemize} \item Take $E = {\cal M}({\cal T})$, the space of all finite Borel measures on a Lusin topological space ${\cal T}$ equipped with the topology of weak convergence. The rule $+$ is the usual addition of measures. Then $(Q_{t}^{\mu}, t \geq 0)$ is a continuous-state branching process with immigration, where the laws of the immigration process are $(\mu_{t}, t \geq 0)$. For a monograph account of these processes, see \cite{Li}. For interesting connections between catalytic branching processes and Mehler semigroups, see \cite{DLSS}; indeed, in that paper it is shown that L\'{e}vy-driven Ornstein-Uhlenbeck processeses on $H = L^{2}((0, \infty))$ arise as limits of fluctuations of the immigration process associated with catalytic branching super-absorbing Brownian motion in $(0, \infty)$.

\item Take $E = \R^{+} \times \R$, and let $(Q_{t}, t \geq 0)$ be a homogeneous affine process. In this case $(Q_{t}^{\mu}, t \geq 0)$ is a general affine process. This is shown in \cite{DL}; applications to finance of general affine processes may be found in \cite{DFS}.

\end{itemize}

Another, more direct generalisation of (\ref{Me2}) and (\ref{scs2}) is to consider two-parameter objects, so instead of a semigroup, we have a two-parameter evolution family of operators $(U(s,t); s \leq t)$ acting in $H$ such that for all $r \leq s \leq t$,
$$ U(r,t) = U(s,t)U(r,s),$$
and we seek a family of probability measures $(\mu_{s,t}, s \leq t)$ such that
$$ \mu_{r,t} = \mu_{s,t} * (U(s,t)\mu_{r,s}).$$
For a detailed investigation of this set-up, see \cite{OuRo}.

\section{Invariant Measures and Operator \\ Self-decomposability}

Let $Y = (Y(t), t \geq 0)$ be the OU process defined by (\ref{OU}), and $(P(t), t \geq 0)$ be the associated Mehler semigroup. We say that $\mu \in {\cal M}_{1}(H)$ is an {\it invariant measure} for $Y$ if for all $t \geq 0, f \in B_{b}(H)$,
\begin{equation} \label{invme}
\int_{H}P_{t}f(x)\mu(dx) = \int_{H}f(x)\mu(dx).
\end{equation}

The set of all invariant measures for $Y$ (which may be empty) is convex, and the extremal points are the ergodic measures for $Y$ (see Chapter 3 of \cite{dPZ1}). Hence if $Y$ has a unique invariant measure, then it is ergodic. It is well-known that if an invariant measure $\mu$ exists for $Y$, and if we choose $Y_{0}$ to have law $\mu$, then $Y$ is a (strictly)-stationary Markov process. In which case, for all $t \geq 0$,
$$ Y_{0} \st Y(t) = S(t)Y_{0} + \int_{0}^{t}S(t-r)dL(r),$$
or equivalently,
\begin{equation} \label{opsd}
\mu = S(t)\mu + \mu_{t},
\end{equation}
where we recall that $\mu_{t}$ is the law of $\int_{0}^{t}S(r)dL(r)$.

Measures that have the property (\ref{opsd}) are called {\it operator self-decomposable}, and simply {\it self-decomposable} in the case where $S(t) = e^{-\lambda t}$ for some $\lambda > 0$; we will have more to say about these later on in this section. For now we quote the following from \cite{App2}, where it appears as Theorem 2.1:

\begin{theorem} The following are equivalent:
\begin{enumerate}
\item $\mu$ is an invariant measure for $Y$;
\item The OU process $Y$ is strictly stationary with ${\cal L}(Y_{0}) = \mu$;
\item $\mu$ is operator self-decomposable  in the sense of (\ref{opsd}).
\end{enumerate}
\end{theorem}

Heuristically, when they exist, invariant measures for a Markov process typically arise as weak limits of the law of the process, corresponding to the dynamical system ``settling down after a suitably large time has passed''. We will make this intuition more precise below; however observe that (\ref{opsd}) is the formal limit of (\ref{scs1}), as $t \rightarrow \infty$.

\subsection{Invariant Measures}

A very detailed analysis of invariant measures for $Y$ was carried out by Chojnowska-Michalik \cite{CM} (see also \cite{FR, App1}). We state her most general results in two theorems:

\begin{theorem} \label{CMth1} If $(\mu_{t}, t \geq \infty)$ converges weakly as $t \rightarrow \infty$, then the limit $\mu_{\infty}$ is an invariant measure for $Y$. Furthermore, any other invariant measure for $Y$ is of the form $\beta * \mu_{\infty}$, where $\beta \in {\cal M}_{1}(H)$ is such that $\beta = S(t)\beta$, for all $t \geq 0$.
\end{theorem}

\begin{theorem} \label{CMth2}  The net $(\mu_{t}, t \geq \infty)$ converges weakly if and only if the following three conditions are satisfied:
\begin{enumerate}
\item $\int_{0}^{\infty}\tr(S(t)QS(t)^{*})dt < \infty$;
\item $\int_{0}^{\infty}\int_{H}(||S(t)||^{2} \wedge 1)\nu(dx)dt < \infty;$
\item $\lim_{t \rightarrow \infty}\int_{0}^{t}\left(S(r)b + \int_{H}S(r)x[{\bf 1}_{B_{1}}(S(r)x) - {\bf 1}_{B_{1}}(x)]\right)\nu(dx)dr$ exists in $H$.
\end{enumerate}
\end{theorem}

The measure $\mu_{\infty}$, defined in Theorem \ref{CMth1}, is infinitely divisible, and its characteristics are $(b_{\infty}, Q_{\infty}, \nu_{\infty})$ where
$$ b_{\infty} = \lim_{t \rightarrow \infty}\int_{0}^{t}\left(S(r)b + \int_{H}S(r)x[{\bf 1}_{B_{1}}(S(r)x) - {\bf 1}_{B_{1}}(x)]\right)\nu(dx)dr,$$

$$ Q_{\infty} = \int_{0}^{\infty}S(t)QS(t)^{*}dt,$$

$$ \nu_{\infty}(B) = \int_{0}^{\infty}\nu_{0}(S_{r}^{-1}(B))dr $$

for each $B \in {\cal B}(H)$.

\vspace{5pt}

There are two special cases of interest that are also treated in \cite{CM}:

\begin{itemize}

\item If the semigroup $(S(t), t \geq 0)$ is {\it stable}, i.e. $\lim_{t \rightarrow \infty}S(t)x = 0$, for all $x \in H$, then an invariant measure $\mu_{\infty}$ exists (and is unique) if and only if $\lim_{t \rightarrow \infty}\int_{0}^{t}S(r)dL(r)$ converges in probability, in which case $\mu_{\infty}$ is the law of this limiting random variable.

\item If the semigroup $(S(t), t \geq 0)$ is {\it exponentially stable}, i.e. $||S(t)|| \leq Ce^{-at}$ for some $C \geq 1, a > 0$, then a sufficient condition for existence of a (unique) invariant measure is that
    $$ \int_{H} \log^{+}(||x||)\nu(dx) < \infty.$$
    In the case where $(S(t), t \geq 0)$ is a contraction semigroup having a bounded generator for which $\lim_{t \rightarrow \infty}S(t)= 0$ in the norm topology on $L(H)$, then this logarithmic integrability condition is both necessary and sufficient, as is shown in Theorem 2.1 of \cite{Jur0}. It is well-known that the condition is both necessary and sufficient in finite dimensions (see \cite{Appbk} p.242, and references therein).

\end{itemize}

\subsection{Operator Self-decomposability}

Let $\mu \in {\cal M}_{1}(H)$ be arbitrary, and define the set:
$$ {\mathbb D}(\mu) = \{T \in L(H); \mu = T\mu * \mu_{T}~\mbox{for some}~\mu_{T} \in {\cal M}_{1}(H)\}.$$

The set ${\mathbb D}(\mu)$ is a semigroup (with respect to composition of operators); indeed, if $S,T \in {\mathbb D}(\mu)$, then $ST \in {\mathbb D}(\mu)$ with
\begin{equation} \label{opfactor}
\mu_{ST} = \mu_{S} * S\mu_{T}.
\end{equation}

It is shown in \cite{Jur1} that ${\mathbb D}(\mu)$ is closed in the strong topology on $L(H)$. It is called the {\it Urbanik semigroup} of the measure $\mu$.
It is easy to see that $\mu$ is operator self-decomposable, in the sense of (\ref{opsd}) if and only if ${\mathbb D}(\mu)$ contains a $C_{0}$-semigroup $(S(t), t \geq 0)$, in which case the measure $\mu_{t} = \mu_{S(t)}$. Then comparing  with (\ref{scs1}), we see that (\ref{opfactor}) is just the skew-convolution property, hence we may associate a Mehler semigroup with $\mu$, and also an OU process by the construction of section 5, if the semigroup is sufficiently well-behaved.

Operator self-decomposable measures arose in the investigation by Urbanik \cite{Urb} of certain normalised sums of independent random variables that are associated with uniformly infinitesimal triangular arrays of probability measures on $H$. The one-dimensional problem was solved by Paul L\'{e}vy (see e.g. Theorem 15.3 in \cite{Sa}, p.51), and gives rise to the classical notion of self-decomposable distribution. Urbanik showed that, under certain technical conditions, operator self-decomposable distributions arise as such limits if and only if $\lim_{t \rightarrow \infty}S(t) = 0$, in the uniform topology on $L(H)$. Under these conditions he was also able to prove infinite divisibility of the limiting laws. It remains an open problem to extend his results to general $C_{0}$-semigroups. More recently, it has been shown \cite{BJ} that self-decomposable distributions arise as limits of triangular arrays of strongly mixing sequences of random variables. For more about operator self-decomposability, see \cite{Jur1}, and references therein.

\vspace{5pt}

If we have an invariant measure $\mu$ for the OU semigroup, it is natural to treat $\mu$ as a ``reference measure'', and to investigate the analytic properties of $P_{t}$ in the Banach spaces $L^{p}(H, \mu)$, where $1 \leq p \leq \infty$. Indeed it is easy to see that $(P_{t}, t \geq 0)$ acts as a contraction semigroup in each of these spaces. It is also a legitimate $C_{0}$-semigroup when $1 \leq p < \infty$ (see the discussion on p.300 of \cite{LR1}).  A useful tool is the {\it second quantised} representation of $P(t)$, given by its action on the chaotic decomposition of $L^{2}(E, \mu)$ into spaces of multiple Wiener-L\'{e}vy integrals. In the non-Gaussian case, this was obtained in \cite{Pesz}, and from a more general viewpoint in \cite{AvN1, AvN2}. In \cite{ORW}, the authors extend the work of \cite{RW} to establish a Harnack inequality for the Mehler semigroup associated to the OU process. This then implies the strong Feller property for the semigroup, from which we can deduce that the transition probabilities of the process are all absolutely continuous with respect to $\mu$.

\section{Cylindrical Ornstein-Uhlenbeck Processes}

Infinite-dimensional processes arise naturally in mathematical modelling through noise that is described, for each $t \geq 0$ by the formal series:
\begin{equation} \label{cyl1}
L(t) = \sum_{n=1}^{\infty}\beta_{n}L_{n}(t)e_{n},
\end{equation}
where $\beta_{n} \in \R$ for each $\nN$, $(L_{n}, \nN)$ is a sequence of independent real valued L\'{e}vy processes, and $(e_{n}, \nN)$ is a complete orthonormal basis for $H$. In fact from now on, we will identify $H$ with the space $l^{2}(\N)$ of all real-valued, square-summable sequences. Such a process $L$, is an example of a {\it cylindrical L\'{e}vy process}, and we give a general definition of this concept below. To see that $L$ is not, in general, an $H$-valued L\'{e}vy process, take $\beta_{n} = 1$ and $L_{n}$ to be a standard Brownian motion for all $\nN$. Then we can easily compute that $X$ has characteristics $(0, I, 0)$; but since the identity $I$ is not trace-class, we do not have a legitimate Brownian motion, or even a L\'{e}vy process.

A systematic theory of cylindrical L\'{e}vy processes was developed in \cite{ApRi}. This was founded on the theory of cylindrical measures which was developed in the 1960s and 70s by Laurent Schwartz, among others (see e.g. \cite{Schw}). For $a_{1}, \ldots, a_{n} \in H$, define $\pi_{a_{1},\ldots, a_{n}}: H \rightarrow \R^{n}$, by $\pi_{a_{1},\ldots, a_{n}}(x) = (\la x, a_{1} \ra, \ldots, \la x, a_{n} \ra)$. Then a {\it cylindrical probability measure} $\mu$ is a set function defined on the algebra generated by all cylinder sets of the form $\pi^{-1}_{a_{1}, \ldots, a_{n}}(B)$ where $B \in {\cal B}(\R^{n}), a_{1}, \ldots, a_{n} \in H$ and $\nN$ such that,
\begin{enumerate}
\item $\mu(H) =1$,
\item The restriction of $\mu$ to the $\sigma$-algebra generated by $\{\pi^{-1}_{a_{1}, \ldots, a_{n}}(B)), B \in {\cal B}(\R^{n})\}$ is a bona fide probability measure, for each fixed $a_{1}, \ldots, a_{n} \in H$ and $\nN$.
\end{enumerate}

\noindent In general, a {\it cylindrical process} $X$ is defined to be a family of linear mappings, $(X(t), t \geq 0)$,  from $H$ to $L^{0}(\Omega, {\cal F}, P)$. Then for each $t \geq 0$ we obtain a cylindrical probability measure $\mu_{t}$ which plays the role of the ``cylidrical law'' of $X(t)$ as follows:
 $$ \mu_{t}(\pi^{-1}_{a_{1}, \ldots, a_{n}}(B)) = P((X(t)a_{1}, \ldots, X(t)a_{n}) \in B).$$ We say that a cylindrical process $X$ is a {\it cylindrical L\'{e}vy process} if for all $t \geq 0, a_{1}, \ldots, a_{n} \in H, \nN$,
$(X(t)a_{1}, \ldots, X(t)a_{n})$ is a L\'{e}vy process in $\R^{n}$. As shown in Lemma 4.2 of \cite{Ried1}, the representation (\ref{cyl1}) gives a specific class of examples of this more general notion when we identify $L(t)$ therein with the mapping which sends each $a \in H$ to the random variable $\sum_{n=1}^{\infty}\beta_{n}L_{n}(t)\la e_{n}, a \ra$, where the $\beta_{n}$'s are chosen to ensure the series converges for all $t \geq 0$. In particular, we obtain a cylindrical L\'{e}vy process in the case where the $L_{n}$'s are i.i.d. centred, square-integrable L\'{e}vy processes and the sequence $(\beta_{n}, \nN)$ is bounded.

A version of the L\'{e}vy-It\^{o} decomposition holds for cylindrical L\'{e}vy processes, and this is used to define stochastic integrals. We may then consider {\it cylindrical Ornstein-Uhlenbeck processes}
$$ Y(t) = S(t)Y_{0} + \int_{0}^{t}S(t-r)dX(r),$$
as solutions of the SDE,

\begin{equation} \label{Langmore}
dY(t) = AY(t)dt + dX(t),
\end{equation}

\noindent where the meaning of the last display is precisely that for all $u \in \mbox{Dom}(A^{*}), t \geq 0$, with probability $1$,
$$ Y(t)(u) = Y_{0}(u) + \int_{0}^{t} Y(r)(A^{*}u) dr + X(t)(u).$$
This is, of course, a natural generalisation to the cylindrical context of the notion of weak solution, as defined in (\ref{weak}). Note that $(Y(t), t \geq 0)$ is, in general, itself a cylindrical process; indeed it is cylindrical Markov in that $(Y(t)a_{1}, \ldots, Y(t)a_{n}), t \geq 0)$ is a Markov process in $\R^{n}$ for each $a_{1}, \ldots, a_{n} \in H, \nN$. Cylindrical notions of Mehler semigroup, invariant measure and selfdecomposable distribution are all developed in \cite{ApRi}.

When working with cylindrical measures and processes, the notion of {\it Radonification} is important. This refers to finding a mapping into a possibly larger space that transforms the cylindrical object into a bona fide one. A important theorem  \cite{JKFR} states that if $(M(t), t \geq 0)$ is a cylindrical semimartingale in $H$, then there exists a Hilbert-Schmidt operator $T$ on $H$ and a semimartingale $(N(t), t \geq 0)$ so that the real-valued processes $(M(t)(T^{*}x), t \geq 0)$ and $ (\la N(t), x \ra , t \geq 0)$ are indistinguishable, for all $x \in H$. In Theorem 5.10 of \cite{Ried1}, conditions are found for a suitable deterministic function $f$ so that its cylindrical stochastic integral $\int_{0}^{t}f(s)dL(s)$ has the property of {\it stochastic integrability} in that there exists a random variable $I_{t}$ such that for all $u \in H, \la I_{t}, u \ra = \left(\int_{0}^{t}f(s)dL(s)\right)(u)$ (see also Corollary 4.4 in \cite{Ried2}).

An alternative approach has been developed in a series of papers that focus on the specific class of cylindrical L\'{e}vy processes defined by (\ref{cyl1}), with the assumption that the $L_{n}$'s are i.i.d. For example \cite{BZ} considers a L\'{e}vy noise obtained by subordinating a cylindrical Wiener process, \cite{PZ1} deals with the case where $L_{n}$ is a symmetric stable process, while  \cite{PZ0} takes $L_{n}$ to be a symmetric pure jump L\'{e}vy process. In the account we give now, we follow \cite{PZ0}; but making a slight extension so that $L_{n}$ is a general symmetric L\'{e}vy process (i.e. we include a Gaussian component). Then for all $y \in \R, \nN, t \geq 0$,
\begin{equation} \label{LKoned}
\E(e^{iyL_{n}(t)}) = \exp{\left\{t\left(-\frac{1}{2}\sigma^{2}y^{2} + \int_{\R}(\cos(yx) - 1)\nu(dx)\right)\right\}},
\end{equation}
where $\sigma \geq 0$ and $\nu$ is a symmetric L\'{e}vy measure. Then $L(t)$ takes values in $H$ for all $t \geq 0$ (with probability one) if and only if
$$ \sum_{n=1}^{\infty}\left\{\beta_{n}^{2}\left(1 + \int_{|x| < 1/\beta_{n}}x^{2}\nu(dx)\right) + \int_{|x| \geq 1/\beta_{n}}\nu(dx)\right\} < \infty.$$

In order to make sense of Ornstein-Uhlenbeck processes within this context, it is necessary to assume that the operator $S(t)$ is compact and self-adjoint for all $t > 0$, and that $(e_{n}, \nN)$ in (\ref{cyl1}) is the complete orthonormal basis of eigenvectors that is guaranteed by the Hilbert-Schmidt theorem. It follows that for each $\nN, t \geq 0$
$$ S(t)e_{n} = e^{-t \gamma_{n}}e_{n},$$
where $\gamma_{n} > 0$ and $\lim_{n \rightarrow \infty}\gamma_{n}=\infty$.
We now consider the infinite-dimensional Langevin equation (\ref{Langmore}) as an infinite sequence of one-dimensional equations
$$ dY_{n}(t) = - \gamma_{n}Y_{n}(t)dt + \beta_{n}dL_{n}(t),$$
having solution starting at $Y(0) = y = (y_{n}) \in H$, given by
$$ Y_{n}(t) = e^{-t \gamma_{n}}y_{n} + \beta_{n}\int_{0}^{t}e^{-\gamma_{n}(t-s)}dL_{n}(t).$$

\noindent Then  $Y(t) = (Y_{n}(t), \nN)$  is a bona-fide stochastic process taking values in $H$ for all $t \geq 0$ (with probability one) if and only if
\begin{equation} \label{Rad}
\sum_{n=1}^{\infty}\frac{\beta_{n}^{2}}{\gamma_{n}} + \sum_{n=1}^{\infty}\frac{1}{\gamma_{n}}\int_{1/\beta_{n}}^{e^{\gamma_{n}}/\beta_{n}}\left(\frac{\psi_{0}(u)}{u^{3}} + \frac{\psi_{1}(u)}{u}\right)du < \infty,
\end{equation}
\noindent where $\psi_{0}(u):= \int_{|x| \leq u}x^{2}\nu(dx)$ and $\psi_{1}(u) = \int_{|x| > u}\nu(dx)$. Furthermore, the process is adapted and Markov. In particular, (\ref{Rad}) is satisfied if the sequence $(\beta_{n})$ is bounded, and the following holds:
$$ \int_{1}^{\infty} \log(y)\nu(dy) < \infty~\mbox{and}~\sum_{n= 1}^{\infty}1/\gamma_{n} < \infty,$$
in which case the process has a unique invariant measure. In Corollary 6.3 of \cite{Ried2}, it is shown that (\ref{Rad}) is a necessary and sufficient condition for the semigroup $(S(t), t \geq 0)$ to be stochastically integrable, (without the assumption that the $L_{n}$s are identically distributed).

The investigation of \cadlag~paths in this context has attracted some attention. In \cite{BGIPPZ}, the authors show that if the vectors $(e_{n}, \nN)$ are all in Dom$(A^{*})$ and the sequence $(\beta_{n})$ fails to converge to zero, then with probability $1$, the OU process $Y$ has no point $t >0$ at which either the right or left limit exists. In \cite{PZ2}, conditions are found on the L\'{e}vy measure for $Y$ to have a \cadlag~modification, provided the semigroup $(S(t), t \geq 0)$ is both exponentially stable and analytic.

In the case where the $L_{n}$'s have symmetric $\alpha$-stable distributions ($0 < \alpha < 2$), a finer analysis may be carried out; in \cite{PZ1} the authors obtain a number of interesting properties of the solution, including gradient estimates on the transition semigroup from which they deduce the strong Feller property; while in \cite{LiZh} it is shown that the OU process $Y$ has a \cadlag~modification if and only if $\sum_{n=1}^{\infty}|\beta_{n}|^{\alpha} < \infty$.

\vspace{5pt}

{\it Acknowledgements.} I am grateful to Zbigniew Jurek, Zenghu Li, Christian Fonseca Mora, Jan van Neerven and Markus Riedle for useful discussions, and help with eliminating typos.

\end{document}